\documentclass[11pt,a4paper]{article}
\usepackage{amsmath,amsfonts,amsthm,amssymb,verbatim,graphicx,color}
\usepackage{authblk}
\usepackage{enumerate}
\usepackage{geometry}
\usepackage{tikz}
\usepackage{hyperref}
\usepackage{lineno}

\usetikzlibrary{shapes.geometric}
\usetikzlibrary{positioning,calc}
\usetikzlibrary{arrows}
\usetikzlibrary{arrows.meta,bending}
\usetikzlibrary{fit}
\tikzset{
    vertex/.style={
        circle,fill=black,scale=0.5
        },
    blank/.style={
        circle,fill=white
        },
    minivertex/.style={
        circle,draw,inner sep=1pt
        },
    bminivertex/.style={
        circle,fill=black,inner sep=1.5pt
        },
    stvertex/.style={
        circle,draw,scale=0.5
        },
    }

\definecolor{rvwvcq}{rgb}{0,0,0}
\definecolor{Gray}{RGB}{160,160,160}
\definecolor{orangeM}{HTML}{D53C00}

\newcommand{\adrian}[1]{\textcolor{red}{#1}}

\newcommand{\ana}[1]{\textcolor{magenta}{#1}}

\title{On the spectrum and expansion of graph associahedra\footnote{Partially supported by French-Argentinian International Research Project SINFIN  (www.irp-sinfin.org).
AG and PT were partially supported by PIP CONICET 1900, PID 80020210300068UR, and PPCT-UNR 80020250400056UR. AP was partially supported by PICT 2020-00549, PIP 11220220100068CO,  PROICO 03-0723, PROINI 03-124, and PROIPRO 03-2923.
\\
An extended abstract presenting some results of this paper appeared in the proceedings of LAGOS 2025 Symposium ~\cite{GPTV-2025}.
}}
\date{ }

\author[1]{Ana Gargantini}
\author[2]{Adri\'an Pastine}
\author[3]{Pablo Torres}
\author[4]{Mario Valencia-Pabon}

\affil[1]{\footnotesize FCEN, Universidad Nacional de Cuyo, Argentina}
\affil[2]{\footnotesize Instituto de Matem\'atica Aplicada San Luis (UNSL-CONICET) and Departamento de Matem\'atica,
Universidad Nacional de San Luis, Argentina}
\affil[3]{\footnotesize FCEIA, Universidad Nacional de Rosario, and CONICET, Argentina}
\affil[4]{\footnotesize Université de Lorraine, CNRS, Inria, LORIA, F-54000 Nancy, France}

\newcommand{\R}{\mathcal{R}}

\newcommand{\A}{\mathcal{A}}

\newcommand{\mc}[1]{\mathcal{#1}}

\newcommand{\SPK}{\textnormal{SPK}}

\newtheorem{lemma}{Lemma}
\newtheorem{theorem}{Theorem}

\newtheorem{corollary}{Corollary}
\newtheorem{remark}{Remark}
\newtheorem{claim}{Claim}

\begin{document}

\maketitle

\begin{abstract}
In this article, we contribute to the spectral analysis of graph associahedra by providing a lower bound for the second largest eigenvalue of $\mathcal{A}(G)$. Furthermore, using equitable partitions, we analyze the spectrum of the stellohedron $\mathcal{A}(K_{1,n})$. Specifically, we prove the existence of an eigenvalue in each interval $(n-i, n-i+1]$ for $1 \leq i \leq 5$, establish the presence of an eigenvalue with high multiplicity in $(n - \frac{3}{n} + \frac{2}{n^2-n}, n)$, and identify two additional small eigenvalues.


\end{abstract}

\section{Introduction}
\label{section_introduction}

The graph associahedron $\A(G)$ of a connected graph $G$ is a convex polytope whose 1-skeleton can be combinatorially described in terms of search trees on $G$ and the rotation operation. More precisely, this 1-skeleton is isomorphic to the rotation graph $\R(G)$ of $G$, which is the graph whose vertices correspond to the search trees on $G$ and whose edges correspond to single tree rotations \cite{CPV-2022a}. Graph associahedra generalize several well-known families of polytopes, including the classical associahedra, cyclohedra, and permutohedra (see \cite{CD-2006, Pos-2009}).

While previous studies have investigated the combinatorial properties of these graphs---such as diameter, chromatic number, and connectivity---by treating the 1-skeleton simply as the graph associahedron $\A(G)$, we explicitly adopt the framework and terminology of the rotation graph $\R(G)$.

Under this framework, the structural analysis of these objects can be extended to their algebraic features. However, the spectral analysis of graph associahedra remains largely undeveloped, with only a few results available for specific families of rotation graphs. For instance, Bacher \cite{bacher1994valeur} determined the explicit eigenvalues of the permutohedron, while Cioabă and Gupta \cite{CG-2022} obtained bounds on the smallest eigenvalue of the classical associahedron (i.e., the rotation graph of a path). Additionally, Krakovski and Mohar \cite{krakovski2012spectrum} proved that the spectrum of $\R(K_n)$ contains all integers from $-(n - 1)$ to $n - 1$ (except $0$ if $n = 2$ or $n = 3$).

A closely related parameter is the edge expansion of these graphs, a property intimately linked to their mixing time and connected to the second largest eigenvalue via the Cheeger's inequalities. Several bounds have recently been established for the expansion of specific families of graph associahedra. For the $n$-dimensional permutohedron, Collares, Doolittle, and Erde \cite{collares2024evolution} established lower and upper bounds of $\Omega(1/n^2)$ and $O(1/n)$, respectively. More recently, Eppstein and Frishberg \cite{EF-2023} proved that the edge expansion of the classical associahedron is $\Omega(\frac{1}{\sqrt{n}\ \log n})$ and $O(\frac{1}{\sqrt{n}})$, and Chang, Defant, and Frishberg \cite{CDF-2025} established analogous bounds for cyclohedra. In the broader context of $0/1$ polytopes, Mihail and Vazirani \cite{FM-1991,Mih-1992} conjectured that all graphs realizable as the $1$-skeleton of a $0/1$-polytope have expansion at least one. This conjecture is known to hold for
several families of $0/1$-polytopes as for instance simple $0/1$-polytopes \cite{VK04} or
base polytopes of matroids \cite{ALGV-2024b} but remains open in general. Recently, Cardinal and Pournin \cite{CP-2024} have shown that this conjecture does not hold for half-integer polytopes.

In this paper, we study the edge expansion and the second largest eigenvalue of $\R(G)$ for a connected graph $G$ on $n+1$ vertices, with $n \ge 2$. We show that $\lambda_2 \ge n-2$ holds for every rotation graph $\R(G)$, and subsequently refine this general lower bound for graphs $G$ containing twin vertices. As a consequence, we deduce that the spectral gap of $\R(G)$ vanishes asymptotically when $G$ is a complete bipartite graph or a complete split graph. Furthermore, Section~\ref{sec:stello} is devoted to the spectrum of the stellohedron $\R(K_{1,n})$: by employing equitable partitions, we prove the existence of an eigenvalue in each interval $(n-i, n-i+1]$ for $1 \le i \le 5$, establish an eigenvalue with high multiplicity in $(n - \frac{3}{n} + \frac{2}{n^2-n}, n)$, and identify two additional small eigenvalues.

\section{Preliminaries}

In this section we present the definitions and results needed throughout the text.

Let $G$ be a graph on $n$ vertices. The adjacency matrix $A_G$ of $G$ is the $n\times n$ matrix with entry $(A_G)_{ij}=1$ if the vertices $i$ and $j$ are adjacent, and $0$ otherwise. We denote the eigenvalues of $A_G$ by $\lambda_1 \geq \lambda_2 \geq \ldots \geq \lambda_{|\mathcal{R}(G)|}$. 

The {\em edge expansion} of graph $G$ is defined as
\[
h(G) = \min_{S\in\Omega}\frac{|\partial S|}{|S|},
\]
with $\Omega=\{S\subseteq V(G) \mid |S| \leq |V|/2\}$, and where for a subset $S\subseteq V(G)$, $\partial S$ is the set of edges with one endpoint in $S$ and the other in $V(G)\setminus S$.

\begin{theorem}[Cheeger's inequalities \cite{Alo-1986, SJ-1989}]
\label{theo-cheeger}
Let $G$ be a $d$-regular graph and let $\lambda_2$ be the second largest eigenvalue of the matrix $A_G$ of $G$. Then,
\[
\frac{d-\lambda_2}{2} \leq h(G) \leq \sqrt{2d(d - \lambda_2)}.
\]
\end{theorem}

A \emph{search tree} $T$ on a connected graph $G$ is a rooted tree with vertex set $V(G)$, defined recursively as follows: the root of $T$ is a chosen vertex $r\in V(G)$, and the children of $r$ are the roots of the search trees on each connected component of $G-r$ \cite{CPV-2022a}.

Let $u, v \in V(T)$ adjacent and such that $v$ is a child of $u$ in $T$. The rotation of $u$ and $v$ (or $uv$-rotation) in $T$ transforms $T$ into a new search tree $T'$ on $G$ satisfying the following conditions:
\begin{enumerate}[(i)]
	\item $u$ is a child of $v$ in $T'$. If $u$ is the root of $T$, then $v$ is the root of $T'$. Otherwise, if $p$ is the parent of $u$ in $T$, then $p$ is the parent of $v$ in $T'$
	\item any subtree $S$ of $u$ in $T$ (distinct from the subtree rooted in $v$) remains a subtree of $u$ in $T'$
	\item each subtree $S$ of $v$ in $T$ is reattached according to adjacency in $G$: if $u$ is adjacent to some vertex of $S$ in $G$, then $S$ becomes a subtree of $u$ in $T'$; otherwise, $S$ remains a subtree of $v$.
\end{enumerate}

The \emph{rotation graph} $\R(G)$ of a connected graph $G$ is defined as the undirected graph whose vertices are the search trees on $G$, and where two search trees are adjacent if one can be obtained from the other by a single rotation operation as described above.

Figure \ref{fig:rotation_graph} shows the rotation graph of the star graph $K_{1,3}$ ($1$-skeleton of the 3-dimensional stellohedron). In this figure we denote by $a_1a_2a_3a_4$ the search tree that is a path with consecutive vertices $a_1,a_2,a_3,a_4$ rooted at $a_1$, by $a_1a_2{\substack{a_3\\a_4}}$ the search tree rooted at $a_1$ that consists in  the path $a_1a_2$ with $a_3$ and $a_4$ attached to $a_2$ as leaves, and, similarly, by $a_1{\substack{a_2\\a_3\\a_4}}$ the search tree rooted at $a_1$ with $a_2$, $a_3$ and $a_4$ attached to $a_1$ as leaves.

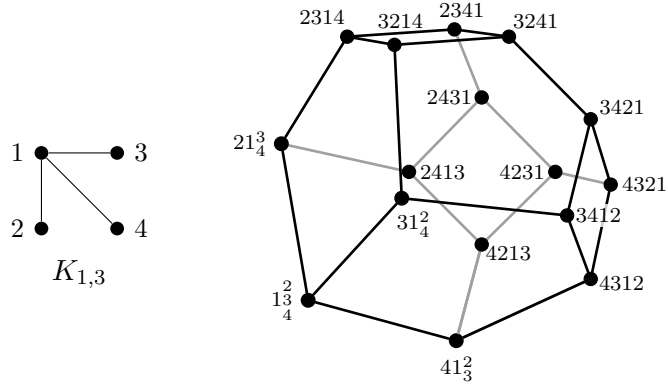
\begin{figure}
    \centering
    \begin{tikzpicture}
        \node[vertex,label=left:{\small $1$}] at (0,1) (1) {};
        \node[vertex,label=right:{\small $3$}] at (1,1) (3) {};
        \node[vertex,label=left:{\small $2$}] at (0,0) (2) {};
        \node[vertex,label=right:{\small $4$}] at (1,0) (4) {};
        \node at (0.5,-0.6) (K13) {$K_{1,3}$};
        \draw (1)--(2);
        \draw (1)--(3);
        \draw (1)--(4);
        \node at (0,-2) (b) {};
    \end{tikzpicture}
    \qquad
    \begin{tikzpicture}[line cap=round,line join=round,>=triangle 45,x=0.6cm,y=0.6cm]
    \draw [line width=1pt,color=Gray] (-0.79,0.58)-- (0.81,-1.02);
    \draw [line width=1pt,color=Gray] (0.81,2.22)-- (-0.79,0.58);
    \draw[line width=1pt,Gray] (-0.79,0.58)--(-3.6,1.20);
\draw [line width=1pt,color=Gray] (2.43,0.58)-- (3.65,0.3);
\draw [line width=1pt,color=Gray] (0.81,2.22)-- (0.21,3.72);
\draw [line width=1pt] (-2.15,3.56)-- (-1.11,3.38);
\draw [line width=1pt] (-1.11,3.38)-- (1.41,3.56);
\draw [line width=1pt] (1.41,3.56)-- (0.21,3.72);
\draw [line width=1pt] (0.21,3.72)-- (-2.15,3.56);

\draw [line width=1pt,color=Gray] (0.81,-1.02)-- (2.43,0.58);
\draw [line width=1pt,color=Gray] (2.43,0.58)-- (0.81,2.22);
\draw [line width=1pt] (3.21,1.74)-- (2.69,-0.38);
\draw [line width=1pt] (3.21,1.74)-- (3.65,0.3);
\draw [line width=1pt] (3.65,0.3)-- (3.21,-1.78);
\draw [line width=1pt] (3.21,-1.78)-- (2.69,-0.38);
\draw [line width=1pt,color=Gray] (0.81,-1.02)-- (0.25,-3.14);
\draw [line width=1pt] (1.41,3.56)-- (3.21,1.74);
\begin{scriptsize}
\draw [fill=rvwvcq] (-1.11,3.38) circle (2.5pt);
\draw[color=rvwvcq] (-0.98,3.9) node {3214};
\draw [fill=rvwvcq] (-2.15,3.56) circle (2.5pt);
\draw[color=rvwvcq] (-2.7,3.99) node{2314};
\draw [fill=rvwvcq] (0.21,3.72) circle (2.5pt);
\draw[color=rvwvcq] (0.37,4.15) node{2341};
\draw [fill=rvwvcq] (1.41,3.56) circle (2.5pt);
\draw[color=rvwvcq] (2,3.9) node{3241};
\draw [fill=rvwvcq] (3.21,1.74) circle (2.5pt);
\draw[color=rvwvcq] (3.9,2) node{3421};
\draw [fill=rvwvcq] (2.69,-0.38) circle (2.5pt);
\draw (3.5,-0.38) node[rectangle,minimum height=0.3cm,fill=white,fill opacity=0.8] { };
\draw[color=rvwvcq] (3.38,-0.38) node {3412};
\draw [fill=rvwvcq] (3.21,-1.78) circle (2.5pt);
\draw[color=rvwvcq] (3.89,-1.88) node{4312};
\draw [fill=rvwvcq] (3.65,0.3) circle (2.5pt);
\draw[color=rvwvcq] (4.4,0.3) node{4321};
\draw [fill=rvwvcq] (2.43,0.58) circle (2.5pt);
\draw[color=rvwvcq] (1.7,0.58) node{4231};
\draw [fill=rvwvcq] (0.81,2.22) circle (2.5pt);
\draw[color=rvwvcq] (0.1,2.22) node{2431};
\node[circle,fill=black,scale=0.7,label=below:{41${\substack{2\\3}}$}] at (0.25,-3.14) (Y) {};
\draw[color=rvwvcq] (-2.3,-4.05) node{};
\node[circle,fill=black,scale=0.7,label=left:{1${\substack{2\\3\\4}}$}] at (-3.01,-2.25) (B) {};
\draw[color=rvwvcq] (-0.7,-1.2) node[label=above:{31${\substack{2\\4}}$}] {};
\node[circle,fill=black,scale=0.7] at (-0.95,0) (G) {};
\node[circle,fill=black,scale=0.7,label=left:{21${\substack{3\\4}}$}] at (-3.6,1.20) (R) {};

\draw[line width=1pt,gray] (3.21,-1.78)--(Y);
\draw[line width=1pt] (B)--(Y);
\draw[line width=1pt] (R)--(B);
\draw[line width=1pt] (B)--(G);
\draw[line width=1pt] (-1.11,3.38)--(G);
\draw[line width=1pt] (2.69,-0.38)--(G);
\draw[line width=1pt] (-2.15,3.56)--(R);

\draw[line width=1pt,Gray] (0.81,-1.02)--(Y); 
\draw[line width=1pt] (3.21,-1.78)--(Y); 
\draw[color=rvwvcq] (-0.05,0.58) node{2413};
\draw[color=rvwvcq,fill=black] (1.4,-1.2) node{4213};
\draw [fill=rvwvcq] (-0.79,0.58) circle (2.5pt);
\draw [fill=rvwvcq] (0.81,-1.02) circle (2.5pt);
\end{scriptsize}

\end{tikzpicture}    
    \caption{The graph $K_{1,3}$ and the corresponding rotation graph $\R(K_{1,3})$.}
    \label{fig:rotation_graph}
\end{figure}

\section{Lower bound on the second largest eigenvalue of graph associahedra}

Notice that if $G$ is a connected graph with $n+1$ vertices, then $\mathcal{R}(G)$ is a $n$-regular graph. Thus, $\lambda_1 = n$ with eigenvector $\vec{1}$. In this section we establish a lower bound on the second largest eigenvalue $\lambda_2$ of the adjacency matrix of $\mathcal{R}(G)$. To this end, we make use of the well known fact that $\lambda_2$ is given by the Rayleigh quotient,
\[
\lambda_2 = \max_{\vec{z} \perp \vec{1}}\frac{\vec{z}^TA\vec{z}}{\vec{z}^T\vec{z}}.
\]

It is also known that any tree on $n\geq 2$ vertices contains at least two different vertices of degree one and that every connected graph contains a spanning tree (see \cite{Bol-2013}). The following remark is immediate.

\begin{remark}
\label{prop-kappas}
Let $G$ be a connected graph on $n > 2$ vertices. Then, $G$ contains at least two distinct vertices $x$ and $y$ such that $G-x$ and $G-y$ are both connected subgraphs of $G$.
\end{remark}

\begin{theorem}
\label{theo-second-eigenvalue}
Let $G$ be a connected graph on $n+1$ vertices with $n\geq 2$. Let $\mathcal{R}(G)$ be the $n$-regular rotation graph of $G$. Let $\lambda_2$ be the second largest eigenvalue of the adjacency matrix $A$ of $\mathcal{R}(G)$. Then, \[\lambda_2 \geq n-2.\]
\end{theorem}

\begin{proof}
By Remark \ref{prop-kappas}, $G$ contains two vertices $x$ and $y$ such that $G -x$ and $G-y$ are both connected subgraphs of $G$. Let $X\subset V(\mathcal{R}(G))$ be the set of search trees of $G$ rooted at vertex $x$ and let $Y\subset V(\mathcal{R}(G))$ be the set of search trees of $G$ rooted at $y$. Consider the vector $\vec{z}$ on $V(\mathcal{R}(G))$ given by
\[
\vec{z}_v=\begin{cases}
1/|X|, &\text{if $v\in X$,}\\
-1/|Y|, &\text{if $v\in Y$,}\\
0, &\text{otherwise.}
\end{cases}
\]

Notice that 
\[
[\vec{z}^T A]_v=\frac{|N(v)\cap X|}{|X|}-\frac{|N(v)\cap Y|}{|Y|}.
\]
Since each vertex in $X$ is adjacent to $n-1$ vertices in $X$ and at most one vertex in $Y$ and,  similarly, each vertex in $Y$ is adjacent to $n-1$ vertices in $Y$ and at most one vertex in $X$, then  $[\vec{z}^T A]_v \geq \frac{n-1}{|X|}-\frac{1}{|Y|}$ for $v\in X$, and $[\vec{z}^T A]_v \leq \frac{1}{|X|}-\frac{n-1}{|Y|}$ for $v\in Y$.

Therefore, using that there are $|X|$ vertices in $X$ and $|Y|$ vertices in $Y$, we get
\begin{align*}
\vec{z}^TA\vec{z}\geq& 
|X|\frac{1}{|X|}\left(\frac{n-1}{|X|}-\frac{1}{|Y|}\right)+|Y|\frac{1}{|Y|}\left(\frac{n-1}{|Y|}-\frac{1}{|X|}\right) =\\
=&\left(\frac{n-1}{|X|}-\frac{1}{|Y|}\right)+\left(\frac{n-1}{|Y|}-\frac{1}{|X|}\right)=\\
=& \frac{n-2}{|Y|}+\frac{n-2}{|X|}
=(n-2)\left(\frac{1}{|X|}+\frac{1}{|Y|}\right)
=(n-2)\frac{|X|+|Y|}{|X||Y|}.
\end{align*}
On the other hand,
\begin{align*}
\vec{z}^T\vec{z}=&|X|\frac{1}{|X|^2}+|Y|\frac{1}{|Y|^2}
=\frac{1}{|X|}+\frac{1}{|Y|}
=\frac{|X|+|Y|}{|X||Y|}.
\end{align*}
Therefore,
\begin{align*}
\frac{\vec{z}^TA\vec{z}}{\vec{z}^T\vec{z}}\geq & n-2.
\end{align*}

As $\vec{1}$ is an eigenvector of $A$ with the largest eigenvalue $n$ and as $\vec{z} \perp \vec{1}$ then, by Rayleigh's principle, the result follows. 
\end{proof}

\begin{remark}
Observe that if the connected component of $G-x$ that contains $y$ has more than one vertex (i.e. if $N_G(y)\neq \{x\}$, so $y$ is not a pendant vertex adjacent to $x$), then there exists a search tree $T$ on $G$ with $r_T=x$ (that is $T\in X$) and $xy\notin E(T)$. Such a tree $T$ has no neighbor in $Y$, and therefore $[\vec z^T A]_T>\frac{n-1}{|X|}-\frac{1}{|Y|}$. In this case the bound $\lambda_2\geq n-2$ is strict. Analogously, the bound is also strict if the connected component of $G-y$ that contains $x$ has more than one vertex.
\end{remark}

\subsection{Better bounds for some classes of graph associahedra}

In this section we bound the edge expansion $h(\R(G))$ from above using pairs of twin vertices. The strategy relies on the structural relationship between $\R(G)$ and $\R(G')$ when $G'$ is obtained from $G$ by adding a twin of a vertex $v\in V(G)$. Specifically, the insertion operation described below lets us split $V(\R(G'))$ into two pieces of comparable size whose boundary we can compute exactly in terms of $\R(G)$. This produces explicit upper bounds for $h$ (and, via Cheeger's inequality, lower bounds on the second eigenvalue $\lambda_2$) for permutohedra, associahedra of split complete graphs, and stellohedra.

The underlying strategy relies on the structural relationship between $\R(G)$ and $\R(G')$ when $G'$ is obtained by adding a twin of a vertex $v\in V(G)$. Specifically, the insertion operation described below allows us to partition $V(\R(G'))$ into two subsets of comparable size, whose edge boundary can be computed exactly in terms of $\R(G)$. This approach yields explicit upper bounds for $h$ (and, via Cheeger's inequality, upper bounds on the second largest eigenvalue $\lambda_2$) for permutohedra, stellohedra, and associahedra of split complete graphs.

Recall that the \textit{open neighborhood} of $v\in V(G)$ is $N_G(v)=\{u\in V(G)\mid uv\in E(G)\}$ and the \textit{closed neighborhood} of $v$ is $N_G[v]=N_G(v)\cup \{v\}$. Two vertices $u,v\in V(G)$ are \textit{false twins} if $N_G(v)=N_G(u)$, and they are \textit{true twins} if $N_G[u]=N_G[v]$.

We recall first the operation of \textit{insertion} of a vertex in a rooted tree. Let $T$ be a rooted tree and $v\in V(T)$. Denote $d_T(v)$ the distance in $T$ between the root of $T$ and $v$. For $i\in\{0,\ldots,d_T(v)\}$, we denote by $T(i, x, v)$ the rooted tree with vertex set $V(T) \cup \{x\}$ constructed as follows.
\begin{itemize}
    \item If $i=0$ (insertion before the root):\\
    $T(0, x, v)$ is the rooted tree with root is $x$ such that $T$ is the only subtree of $x$.
    \item If $1 \leq i \leq d_T(v)$ (insertion along the path from $r_T$ to $v$):\\
    Suppose the path from $r_T$ to $v$ has vertices $r_T=a_0,\ldots,a_{d_T(v)}=v$. Then $T(i,x,v)$ is the rooted tree with root $r_T$ obtained from $T$ by subdividing the edge $a_{i-1}a_i$ and labeling with $x$ the vertex added in the subdivision.
\end{itemize}

We denote $\R(G',v',v)$ the subgraph of $\R(G')$ induced by the set $\{T(i,v',v)\mid T\in\R(G), 0\leq i\leq d_T(g)\}$.

\begin{remark}[\cite{gargantini2026effect}]\label{rem:border_edges}

In a tree $T(i,v',v) \in \R(G',v',v)$, $v$ is a 
descendant of $v'$. Any $e$-rotation such that $e \neq vv'$ preserves this relationship between $v'$ and $v$. Thus, for every tree $T(i,x,v) \in \R(G',v',v)$, each $e$-rotation  with $e \neq vv'$ produces a search tree that belongs to $\R(G',v',v)$. In particular, if $i < d_T(v)$, the closed neighborhood of $T(i,x,v)$ in $\R(G')$ is completely contained in $\R(G',v',v)$, whereas the neighborhood of $T(d_T(v),v',v)$ has exactly one element outside it, obtained from $T(d_T(v),v',v)$ by applying a $vv'$-rotation, since in that tree $v'$ is a child of $v$.
\end{remark}

\begin{lemma}[Lemma 2.4(c) from \cite{gargantini2026effect}]
\label{lemma_ins_is_ST}
    Let $G'$ be a connected graph and let $x,v\in V(G)$ such that $N_{G'}[x]\subseteq N_{G'}[v]$. Denote $G=G'-x$. Let $T$ be a search tree on $G$ and $i\in\{0,\ldots,d_T(v)\}$. Then, $T(i,x,v)$ is a search tree on $G'$.
\end{lemma}

\begin{remark}[\cite{gargantini2026effect}]
\label{rem:notation}
Let $G'$ a connected graph and $v,v'\in V(G)$ a pair of true twins. Let $\rho\colon V(G')\rightarrow V(G')$  be the map defined by $\rho(v)=v'$, $\rho(v')=v$ and $\rho(x)=x$ if $x\neq v,v'$. Notice that $\rho$ is a graph automorphism. Moreover it induces a graph automorphism $\rho^\ast\colon \R(G')\rightarrow \R(G')$, where $\rho^\ast(T)$ is the search tree on $G'$ with vertex set $V(T)$ and edge set $\{\rho(x)\rho(y)\mid xy\in E(T)\}$.
\end{remark}

Let $G$ be a connected graph and $v\in V(G)$. We denote by $G^t_v$ the graph obtained from $G$ by adding a vertex $v'$ such that $v$ and $v'$ are true twins.

\begin{lemma}[Proposition 3.7(a) from \cite{gargantini2026effect}]
\label{lemma_halves_of_Rtt}
   Let $G$ be a connected graph and $v\in V(G)$. $\R(G_v^t,v',v)$ and $\R(G_v^t,v,v')$ are isomorphic and disjoint induced subgraphs of $\R(G_v^t)$.
\end{lemma}

\begin{theorem}
\label{theo_expansion_tt}
Let $G$ be a connected graph and $v\in V(G)$. Then,
\[
h(\R(G_v^t))\leq \frac{2|\R(G)|}{|\R(G_v^t)|}.
\]
\end{theorem}

\begin{proof}
The strategy is to split $V(\R(G^t_v))$ into two isomorphic halves via the automorphism $\rho^\ast$ of Remark \ref{rem:notation}, and then bound the edge boundary of one half using the border description of Remark \ref{rem:border_edges}.

Let $\mathcal R(G^t_v,v',v)$ be the set of search trees on $G_v$ such that $v$ is a descendant of $v'$, and $\mathcal R(G^t_v,v,v')$ the set of search trees on $G_v$ such that $v'$ is a descendant of $v$. From Lemma \ref{lemma_halves_of_Rtt}, follows that $\{\mathcal R(G^t_v,v',v),\mathcal R(G^t_v,v,v')\}$ is a partition of $V(\R(G_v))$ and moreover, $|\mathcal R(G^t_v,v',v)|=|\mathcal R(G^t_v,v,v')|$. This implies that
\[
h(\R(G_v^t))\leq\displaystyle\frac{|\partial \mathcal R(G^t_v,v',v)|}{|\mathcal R(G^t_v,v',v)|}=\frac{2|\partial \mathcal R(G^t_v,v',v)|}{|\R(G_v^t)|}.
\]

From Remark \ref{rem:border_edges} follows that $\partial \R(G^t_v,v',v)$ is in one-to-one correspondence with $\R(G)$, and therefore,
\[
h(\R(G_v^t))\leq\displaystyle\frac{2|\R(G)|}{|\R(G_v^t)|}.
\]
\end{proof}

Since complete graphs and split complete graphs possess pairs of true twins, we can apply the previous theorem to obtain upper bounds for the edge expansion of permutohedra and associahedra of split complete graphs. 

Applying the formula to the permutohedron, recalling that  $|\R(K_n)|=n!$ for every positive integer $n$, this implies the bound $h(\R(K_n))\leq \frac{2}{n}$ given by Collares et al. \cite{collares2024evolution}, and as a consequence, if $\lambda_2$ is the second largest eigenvalue of the adjacency matrix of $\R(K_n)$ then, $\lambda_2\geq n-1-\frac{4}{n}$, by the left Cheeger's inequality (\ref{theo-cheeger}).

In the case of of split complete graphs, let $G=\operatorname{SPK}_{p,q}$, and $G'=\operatorname{SPK}_{p+1,q}$. Let $P$ be the clique of $G$, and let $x$ be the vertex that has been added to the clique in $G'$. 

Now, given a search tree $T\in \R(G)$,  we have that for every $v$ in the clique $P$, $T(d_T(v),x,v)\in V(\R(G'))$. This gives at least $|P|=p$ search trees in $\R(G')$ for each search tree in $\R(G)$. Notice that removing $x$ from the tree returns $T$, thus different trees in $\R(G)$ give different trees in $\R(G)$. Therefore, 
\[
|\R(G')|>  p|\R(G)|,
\]
where the inequality is strict since for a search tree $T\in V(\R(G))$ such that the root $r_T$ is in the independent set of $G$, then $T(0,x,r_T)$ is a search tree on $G'$ not counted in the previous argument. From this and Theorem \ref{theo_expansion_tt} we obtain the following corollary.

\begin{corollary}
Let $p,q\in\mathbb N$ and let $\lambda_2$ be the second largest eigenvalue of $\R(\SPK_{p+1,q})$. Then
\begin{enumerate}[(a)]
    \item $h(\R(\SPK_{p+1,q}))< \displaystyle\frac{2}{p}$, and 
    \item $\lambda_2>p+q-\frac{4}{p}$.
\end{enumerate}
\end{corollary}

%

A similar bound holds when we add a false twin instead of a true twin, but the argument requires an extra correction term: unlike in the true-twin case, not every tree obtained by inserting $v'$ is a search tree on $G^f_v$, since some of the trees produced this way have $v$ and $v'$ adjacent with one of them a leaf. Given a connected graph $G$ and $v\in V(G)$, we denote by $G^f_v$ the graph obtained from $G$ by adding a vertex $v'$ such that $v$ and $v'$ are false twins.

For $v,v'$ true twins in $G_v^t$, let $\mathcal{S}(G^t_v)$ denote the set of search trees on $G^t_v$ in which $v$ and $v'$ are adjacent and one of them is a leaf. For $v,v'$ false twins in $G_v^f$, let $\mathcal T_\land(G)$ be the set of search trees $\tilde T\in\R(G_v^f)$ such that $v$ and $v'$  are both leaves of $\tilde T$.

\begin{lemma}[Propositions 3.8 and 3.9 from \cite{gargantini2026effect}]
\label{lemma:structure_of_ft}
Let $G$ be a connected graph and $v\in V(G)$. Then
\begin{enumerate}[(a)]
\item $\{V(\R(G_v^t))-\mathcal S(G_v^t),\ \mathcal T_\land(G)\}$
is a partition of $V(\R(G_v^f))$.
\item\label{item_subgraph_ft} The subgraphs $\R(G_v^f)-\mathcal T_\land(G)\subseteq\R(G_v^f)$ and $\R(G_v^t)-\mathcal S(G_v^t)\subseteq\R(G_v^t)$ are isomorphic.
\item If $T_\land$ and $T'_\land$ are two search trees from $\mathcal T_\land(G)$ that are adjacent in $\R(G_v^f)$, then $T$ and $T'$ are adjacent in $\R(G)$.
\item If $T_\land\in\mathcal T_\land(G)$, then its neighbors in $\R(G_v^t)-\mathcal S(G_v^t)$ are $T(d_T(v)-1,v',v)$ and $T(d_T(v)-1,v,v')$.
\end{enumerate}
\end{lemma}

\begin{theorem}
\label{theo_expansion_ft}
    Let $G$ be a connected graph such that $|G|\geq 2$ and $v\in V(G)$. Then, \[h(\R(G^f_v)) < \displaystyle\frac{2|\R(G)|}{|\R(G^f_v)|-|\R(G)|}.\]
\end{theorem}

\begin{proof}
By Lemma \ref{lemma:structure_of_ft}(a), the rooted trees in $V(\mathcal{R}(G^t_v))-\mathcal{S}(G^t_v)$ are search trees on $G^f_v$, and together with $\mathcal{T}_\wedge$ they partition $V(\mathcal{R}(G^f_v))$.

Set $R=V(\mathcal{R}(G^t_v,v',v))-\mathcal{S}(G^t_v))$. Note that $R$ and $V(\mathcal{R}(G^t_v,v,v'))-\mathcal{S}(G^t_v))$ induce isomorphic subgraphs of $\mathcal{R}(G^f_v)$, so $|\mathcal{R}(G^f_v)|=2|R|+|\mathcal{T}_\wedge|$.

By Lemma \ref{lemma:structure_of_ft}(d) of~\cite{gargantini2026effect}, every $T_\wedge\in\mathcal{T}_\wedge$ is adjacent in $\mathcal{R}(G^f_v)$ to exactly the two trees $T(d_T(v)-1,v',v)$ and $T(d_T(v)-1,v,v')$, both lying in $V(\mathcal{R}(G^f_v))-\mathcal{T}_\wedge$. Hence every edge in $\partial R$ corresponds to a unique element of $\mathcal{T}_\wedge$, giving $|\partial R|=|\mathcal{T}_\wedge|$.

The assignment $T_\wedge\mapsto T$ (where $T$ is the search tree on $G$ obtained by removing $v'$) is a bijection from $\mathcal{T}_\wedge$ to the set of search trees on $G$ in which $v$ is a leaf. Since $|G|\geq 2$, not every search tree on $G$ has $v$ as a leaf, so $|\mathcal{T}_\wedge|<|\mathcal{R}(G)|$. Therefore,
\[
h(\mathcal{R}(G^f_v))
\leq \frac{|\partial R|}{|R|}
= \frac{2\,|\mathcal{T}_\wedge|}{|\mathcal{R}(G^f_v)| - |\mathcal{T}_\wedge|}
< \frac{2\,|\mathcal{R}(G)|}{|\mathcal{R}(G^f_v)| - |\mathcal{R}(G)|}. 
\]
\end{proof}

For $n\geq 2$, the star $K_{1,n}$ can be obtained from $K_{1,n-1}$ by adding a false twin to any leaf. Before applying Theorem~\ref{theo_expansion_ft}, we establish the recurrence $|\mathcal{R}(K_{1,n})|=n|\mathcal{R}(K_{1,n-1})|+1$. Every search tree on $K_{1,n}$ either has the universal vertex as a non-root vertex — in which case exactly one of the $n$ leaves is the root, and removing it gives a search tree on $K_{1,n-1}$, yielding $n|\mathcal{R}(K_{1,n-1})|$ such trees — or has the universal vertex as root, with all leaves as non-root vertices; the latter is unique. This gives $|\mathcal{R}(K_{1,n})|=n|\mathcal{R}(K_{1,n-1})|+1$.

Substituting $G=K_{1,n-1}$ and $G^f_v=K_{1,n}$ into Theorem \ref{theo_expansion_ft} together with this recurrence gives the following corollary. A stronger version of this result, obtained through a detailed study of stellohedra, is given in Section~\ref{sec:eigenstello}.

\begin{corollary}
\label{coro:stellos}
    Let $n\geq 2$ and let $\lambda_2$ be the second largest eigenvalue of $\R(K_{1,n})$. Then 
    \begin{enumerate}
        \item $h(\R(K_{1,n}))< \frac{2}{n-1}$.
        \item $\lambda_2> n-\frac{4}{n-1}$.
    \end{enumerate}

\end{corollary}

\begin{proof}
Item 1 follows from Theorem \ref{theo_expansion_ft},
\[
    h(\R(K_{1,n}))< \frac{2|\R(K_{1,n-1})|}{|\R(K_{1,n})|-|\R(K_{1,n-1})|} = \frac{2|\R(K_{1,n-1})|}{n|\R(K_{1,n-1})|+1-|\R(K_{1,n-1})|}\leq \frac{2}{n-1}.
    \]
Item 2 follows from Item 1 via Cheeger's inequality \ref{theo-cheeger}.
\end{proof}

Let $G=\operatorname{SPK}_{p,q}$, $G'=\operatorname{SPK}_{p,q+1}$, and let $Q$ be the independent set of size $q+1$ of $G'$. For each $y\in Q$, let $\mc{T}_y$ be the set of search trees in $\mc{R}(G')$ that have $y$ as its root. Notice that $G'-y$ is isomorphic to $G$, and that $|\mc{T}_y|=|\mc{R}(G'-y)|=|\mc{R}(G)|$, as inserting $y$ as a root
of a search tree in $\mc{R}(G'-y)$ gives a search tree in $\mc{T}_y$, and vice versa. 
Therefore, 
\[
|\mc{R}(G')|\geq \sum_{y\in Q}|\mc{T}_y|=(q+1)|\mc{R}(G)|.
\]
Theorem \ref{theo_expansion_ft} then gives the following corollary.

\begin{corollary}
Let $p,q\in\mathbb N$ and let $\lambda_2$ be the second largest eigenvalue of $\R(\SPK_{p,q+1})$. Then
\begin{enumerate}[(a)]
    \item $h(\R(\SPK_{p,q+1}))< \displaystyle\frac{2}{q}$, and 
    \item $\lambda_2>p+q-\frac{4}{q}$.
\end{enumerate}
\end{corollary}

An argument similar to the one used for $\operatorname{SPK}_{p,q}$ and $\operatorname{SPK}_{p,q+1}$, can be used for $K_{p,q}$ and $K_{p,q+1}$, since two vertices from the maximal independent set of $K_{p,q+1}$ are false twins.
In particular this implies that, as $q\to\infty$ with $p$ fixed, the gap $\lambda_1-\lambda_2$ between the first and second eigenvalues of $\R(K_{p,q})$ goes to $0$.

\section{On the eigenvalues of stellohedra}\label{sec:stello}

This section is dedicated to present some progress regarding the spectrum of stellohedra. We begin by determining the characteristic polynomials related to an equitable partition of the vertices of stellohedra. This approach enables us to establish the presence of at least one eigenvalue of the adjacency matrix of the $n$-regular stellohedron within the interval $(n-i, n-i+1]$ for $1 \leq i \leq 5$. We also give two smaller eigenvalues of the stellohedron.

\subsection{Some structural properties of stellohedra}

Stellohedra exhibit a symmetric structure, which we describe in this section. We make use of this symmetry to study their characteristic polynomial via equitable partitions.

Let $n > 0$ be an integer and let $K_{1,n}$ denote the star graph with $n$ leaves. We denote by $\{x\}\cup[n]$ the vertex set of $K_{1,n}$, where $x$ denotes its only cut vertex and the leaves have labels on $[n]$. For simplicity, we denote by $\mathcal{S}_n=\mathcal{R}(K_{1,n})$. 

It is well known that all search trees on $K_{1,n}$ are brooms (see \cite{CPV-2024}). Let $k$ be an integer with $0 \leq k \leq n$. We denote by $N_k$ the subset of vertices in $\mathcal{S}_n$ where its corresponding brooms have a handle of length equal to $k$ (see Figure \ref{figroot1new}(a) for an example).

\begin{figure}[htp]
 \centering
 \includegraphics[scale=.3]{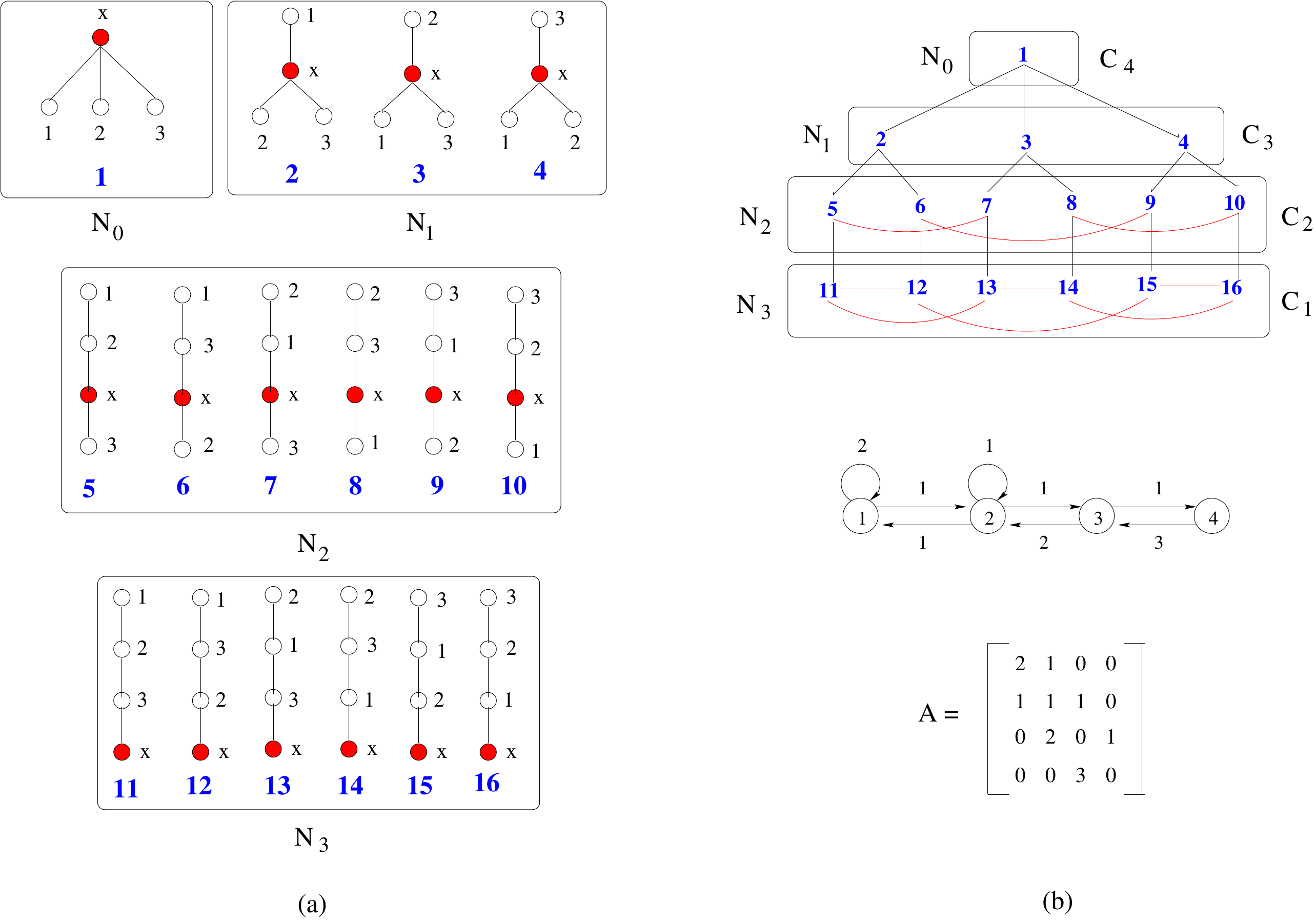}
  \caption{$(a)$ Subsets $N_k$ of $\mathcal{S}_3$, for $0 \leq k \leq 3$. $(b)$ The $1$-skeleton of $\mathcal{S}_3$ with partition $\pi = C_1,\ldots,C_4$ in the top, the quotient graph $\mathcal{S}_3/\pi$ in the middle, and its adjacency matrix $A = A(\mathcal{S}_3/\pi)$ in the bottom.}
 \label{figroot1new}
\end{figure} 

\begin{remark}
\label{obsroot1}
The sets $N_k$, with $0 \leq k \leq n$, have the following properties:
\begin{enumerate}
\item The subgraphs induced by $N_0$ and $N_1$ are both $0$-regular.
\item For $2 \leq k \leq n$, the subgraphs induced by the sets $N_k$ are $(k-1)$-regular.
\item For $0\leq k< n$, each vertex in $N_k$ has exactly $n-k$ neighbors in the set $N_{k+1}$.
\item For $0 < k \leq n$, each vertex in $N_k$ has exactly one neighbor in the set $N_{k-1}$.
\end{enumerate}
\end{remark}

Let $G$ be a graph. Let $\pi$ be a partition of the vertices of $G$ and let $C_1,\ldots,C_r$ be the cells of this partition. This partition is {\em equitable} if the number of vertices in cell $C_j$ that are adjacent to a vertex in $C_i$ is determined only by $i$ and $j$, and not by the specific vertex from $C_i$ used. If $\pi$ is an equitable partition of $G$, the {\em quotient graph} $G/\pi$ is defined as the directed multi-graph with the cells of $\pi$ as its vertices, and if a vertex in $C_i$ has exactly $\nu$ vertices in $C_j$, then $G/\pi$ has $\nu$ arcs from $C_i$ to $C_j$. Notice that if $G$ is a regular graph, then the induced subgraph by each cell $C_i$ must also be regular. In general $G/\pi$ will have loops and parallel arcs. 

The following remark follows directly from Remark \ref{obsroot1}.
\begin{remark}
\label{obsroot2}
Let $n >0$ and let $0 \leq k \leq n$. The partition $\pi = C_1,\ldots,C_{n+1}$ of $\mathcal{S}_n$, where $C_i = N_{n-i+1}$, for $1 \leq i \leq n+1$, is an equitable partition of $\mathcal{S}_n$.
\end{remark}

Figure \ref{figroot1new}(b) gives an example of the equitable partition $\pi$ of $\mathcal{S}_3$ and of the corresponding quotient multi-graph $\mathcal{S}_3/\pi$. The following result is given in \cite{GM-2016}.
\begin{lemma}[Lemma $2.2.2$ and Corollary $2.2.3$ in \cite{GM-2016}]
\label{lemaroot1}
If $\pi$ is an equitable partition of the vertices of a graph $G$, the characteristic polynomial of $G/\pi$ divides the characteristic polynomial of $G$. Consequently, the eigenvalues of $G/\pi$ are eigenvalues of $G$.
\end{lemma}

\subsection{Characteristic polynomial of $\mathcal{S}_n/\pi$}
Let $n>0$ and let $\pi = C_1,\ldots,C_{n+1}$ the equitable partition of $\mathcal{S}_n$ computed in Remark \ref{obsroot2}. Notice that the adjacency matrix of $\mathcal{S}_n/\pi$ is a tridiagonal $(n+1)\times (n+1)$ matrix of the form

\[
A = A(\mathcal{S}_n/\pi) = 
   \begin{pmatrix}
    n-1 &  1  &   0    &   0     & \dots &   &  0   \\
     1  & n-2 &   1    &   0     &       &   &      \\
        &  2  & \ddots & \ddots  &       &   & \vdots\\
        &     & \ddots &   2     &   1   &   &       \\
        &     &        &  n-2    &   1   & 1 &       \\
 \vdots &     &        &         &  n-1  & 0 &  1    \\
      0 &     & \dots  &         &       & n &  0
  \end{pmatrix}
\]

In order to compute the eigendecomposition of $A$ we will use a similarity transformation to construct a symmetric tridiagonal matrix $B = D^{-1}AD$, with $D := \text{diag}(\delta_1,\delta_2,\ldots,\delta_{n+1})$, where $\delta_i = \sqrt{(i-1)!}$ for $i = 1, 2,\ldots,n+1$. Clearly, $A$ and $B$ have the same characteristic polynomial and so the same eigenvalues:

\begin{center}
$
B = 
   \begin{pmatrix}
    n-1 &  1       &   0       &   0        & \dots      &            &  0   \\
     1  & n-2      & \sqrt{2}  &   0        &            &            &      \\
        & \sqrt{2} & \ddots    & \ddots     &            &            & \vdots\\
        &          & \ddots    &   2        & \sqrt{n-2} &            &       \\
        &          &           & \sqrt{n-2} &   1        & \sqrt{n-1} &       \\
 \vdots &          &           &            & \sqrt{n-1} &   0        & \sqrt{n}\\
      0 &          & \dots     &            &            & \sqrt{n}   &  0
  \end{pmatrix}
$
\end{center}

Now, in order to compute the characteristic polynomial of $B$, we need to compute $\text{det}(B - \lambda I_{n+1})$. For $1 \leq k \leq n+1$, let $B_k$ be the principal $k \times k$ submatrix of $B$ formed by the intersection of the first $k$ columns and the first $k$ rows of $B$. Let $f_k(\lambda)$ be the characteristic polynomial of $B_k$, that is, $f_k(\lambda) = \text{det}(B_k - \lambda I_k)$. Clearly, the characteristic polynomial of $B$ is the polynomial $f_{n+1}(\lambda)$. Notice that if we expand the determinant of $B - \lambda I_{n+1}$ by minors, using the last row, and then expanding the last column of one of the reduced determinants, we obtain:
\begin{equation}
f_{n+1}(\lambda) = (-\lambda)f_n(\lambda)- nf_{n-1}(\lambda).
\end{equation}

By performing the same procedure as in equation $(1)$, we obtain a recursive relationship for the polynomials $f_k(\lambda)$, for $0 \leq k \leq n$, where we introduce $f_0(\lambda) = 1$. Then,
\begin{align}
\label{eqroot2}
\begin{split}
f_0(\lambda) & = 1,\\
f_1(\lambda) & = n-1-\lambda,\\
f_k(\lambda) & = (n-k-\lambda)f_{k-1}(\lambda) - (k-1)f_{k-2}(\lambda), \text{ for $2\leq k \leq n$.}
\end{split}
\end{align}

\begin{theorem}
\label{teoroot1}
Let $n>0$ and let $0 \leq k \leq n$. Then, the polynomial $f_k(\lambda)$ in the recurrence equation $(2)$ has the closed form:
$$
f_k(\lambda) = \sum_{p=0}^k \binom{k}{p}(-1)^p\frac{(n-\lambda)!}{(n-\lambda-k+p)!}.
$$
\end{theorem}
\begin{proof}
Let $F(x) = \displaystyle\sum_{k\geq 0}f_k(\lambda)\frac{x^k}{k!}$ be the exponential generating function associated with recurrence $(2)$, and let $F'(x) = \displaystyle\sum_{k \geq 0}f_{k+1}(\lambda)\frac{x^k}{k!}$ its first derivative. Then, equation $(2)$ can be expressed as a formal power series as follows: 
\begin{align*}
\sum_{k\geq 0}f_{k+2}(\lambda)\frac{x^{k+1}}{(k+1)!} &= (n - 1 -\lambda)\sum_{k\geq 0}f_{k+1}(\lambda)\frac{x^{k+1}}{(k+1)!} - \sum_{k\geq 0}f_{k+1}(\lambda)\frac{x^{k+1}}{k!} - \sum_{k\geq 0}f_k(\lambda)\frac{x^{k+1}}{k!}\\
F'(x)-f_1(\lambda) &= (n-1-\lambda)[F(x) - f_0(\lambda)] - xF'(x) - xF(x)\\
F'(x)[x+1] &= [(n-1-\lambda) -x]F(x)\\
F'(x) &= \left[\frac{n-\lambda}{x+1} - 1\right]F(x)
\end{align*}
The general solution of this linear homogeneous differential equation is $F(x) = e^P$, where $P = \int \left[\frac{n-\lambda}{x+1} - 1\right]\,dx$. Therefore, $F(x) = e^{-x}(x+1)^{n-\lambda}$. Now, let $G(x) = \displaystyle\sum_{k\geq 0}g_k\frac{x^k}{k!}$ and let $H(x) = \displaystyle\sum_{k\geq 0}h_k\frac{x^k}{k!}$ two power series. It is well known that the product $G(x)H(x) = \displaystyle\sum_{k \geq 0}\left(\sum_{p=0}^k \binom{k}{p}g_ph_{k-p}\right)\frac{x^k}{k!}$. Therefore, setting $G(x) = e^{-x} = \displaystyle\sum_{k\geq 0}(-1)^k\frac{x^k}{k!}$ and $H(x) = (x+1)^{n-\lambda} = \displaystyle\sum_{k\geq 0}\binom{n-\lambda}{k}x^k = \sum_{k\geq 0} \frac{(n-\lambda)!}{(n-\lambda - k)!}\frac{x^k}{k!}$, the result holds. 
\end{proof}

Finally, using equation $(1)$ and Theorem \ref{teoroot1}, we can easily compute the characteristic polynomial of matrix $B$, that is, the polynomial $f_{n+1}(\lambda)$ as follows:

\begin{theorem}
\label{teoroot2}
Let $n>0$. Then, the polynomial $f_{n+1}(\lambda)$ in the recurrence equation $(1)$ has the closed form:
$$
f_{n+1}(\lambda) = \sum_{p=0}^n(-1)^p\binom{n}{p}(p-\lambda)\frac{(n-\lambda)!}{(p-\lambda)!}.
$$
\end{theorem}
As a direct corollary of Theorem \ref{teoroot1} and \ref{teoroot2} we have the following result.

\begin{corollary}
\label{cororoot1}
Let $n>0$ and let $0 \leq k \leq n+1$. Then, the polynomials $f_k(\lambda)$ can be expressed as:
\begin{align*}
f_0(\lambda) &= 1\\
f_k(\lambda) &= (-1)^k + \sum_{p=0}^{k-1}(-1)^p\binom{k}{p}\prod_{t=0}^{k-(p+1)}(n-\lambda -t), \hspace{1.0cm}\text{for $1 \leq k \leq n$}\\
f_{n+1}(\lambda) &= \sum_{p=0}^{n}(-1)^p\binom{n}{p}\prod_{t=0}^{n-p}(n-\lambda -t)
\end{align*}
\end{corollary}

\subsection{Localizing eigenvalues of $\mathcal{S}_n/\pi$}\label{LocEigen}

In this section, we will use the formulas for the polynomials $f_k(\lambda)$ given in Corollary \ref{cororoot1}. First, notice that $f_{n+1}(n) = 0$ and so, $n$ is an eigenvalue of the adjacency matrix of $\mathcal{S}_n/\pi$. In fact, $\mathcal{S}_n$ is a $n$-regular graph and so $n$ is its maximum eigenvalue. 
Notice that the matrix $B$ is a real tridiagonal symmetric matrix where the entries of the lower (and upper) diagonal are all different from $0$. This kind of matrices has a special property: the sequence $\{f_0(\lambda), f_1(\lambda),\ldots,f_{n+1}(\lambda)\}$ forms a Sturm sequence of polynomials. Let $\lambda = b$ and let $\{\text{sign}(f_0(b)), \text{sign}(f_1(b)),\dots,\text{sign}(f_{n+1}(b))\}$ be the sequence of signs of each polynomial $f_k$ evaluated at point $b$, $0 \leq k \leq n+1$, where if some $f_k(b) = 0$, then choose the sign of $f_k(b)$ to be opposite to that of $f_{k-1}(b)$ (it can be shown that $f_k(b) = 0$ implies $f_{k-1}(b)\neq0$). Moreover, let $s(b)$ denote the number of agreements of sign between consecutive members of the sign sequence. Then,
\begin{theorem}[See Chapter VII in \cite{Usp-1948}]
\label{teoroot3}
The number of roots greater than $\lambda = a$ is given by $s(a)$. For $a < b$, the number of roots in the interval $a < \lambda \leq b$ is given by $s(a) - s(b)$.
\end{theorem}

By Theorem \ref{teoroot3}, we are able to localize the first five largest eigenvalues of $\mathcal{S}_n/\pi$ as follows.

\begin{theorem}
\label{lemaroot2}
Let $n>1$ and let $\beta_i$ be the $i$th largest eigenvalue of the adjacency matrix $A(\mathcal{S}_n/\pi)$, with $1 \leq i \leq 5$. Then, $\beta_i \in (n-i,n-i+1]$.
\end{theorem}

\begin{proof}
We will prove that $s(n-i) = i$, for $i \in \{0,1,2,3,4,5\}$, and thus the result follows from Theorem \ref{teoroot3}. Using the formulas for the polynomials $f_k(\lambda)$ given in Corollary \ref{cororoot1}, we obtain that:
\begin{align*}
f_0(n-i) &= 1,\\
f_k(n-i) &= (-1)^k + \sum_{p=0}^{k-1}(-1)^p\binom{k}{p}\prod_{t=0}^{k-(p+1)}(i -t), \hspace{1.0cm}\text{for $1 \leq k \leq n$},\\
&=\begin{cases}
(-1)^k + \sum_{p=0}^{k-1} (-1)^p\binom{k}{p}i!/(i+p-k)!, &\text{if $1 \leq k \leq i$}\\
\\
(-1)^k + \sum_{p=k-i}^{k-1} (-1)^p\binom{k}{p}i!/(i+p-k)!, &\text{if $i+1 \leq k \leq n$}\\
   \end{cases}\\
f_{n+1}(n-i) &= \sum_{p=0}^{n}(-1)^p\binom{n}{p}\prod_{t=0}^{n-p}(i -t).\\
&=\sum_{p=n+1-i}^{n}(-1)^p\binom{n}{p}i!/(i+p-(n+1))!
\end{align*}
In fact, notice that for $f_k(n-i)$ with $k>i$, if $p \in \{0,\ldots,k-i-1\}$ then $k-(p+1) >= i$ and so, there is $t\in \{0,\ldots,k-(p+1)\}$ such that $t=i$ and thus $\prod(i-t) = 0$. The same behavior occurs in $f_{n+1}(n-i)$ when $p < n+1-i$.
\begin{claim}
\label{cl1}
$(a)$ If $i \in \{1,2\}$ then, the sign sequence $\{\text{sign}(f_k(n-i)) : i+2 \leq k \leq n+1\}$ is full alternating. $(b)$ If $i \in \{3,4,5\}$ then, the sign sequence $\{\text{sign}(f_k(n-i)) : i+3 \leq k \leq n+1\}$ is full alternating. 
\end{claim}
\begin{proof}
For $i=1$ we have that for $3 \leq k \leq n$, $f_k(n-1) = (-1)^k + (-1)^{k-1}k$ and so, $f_k(n-1) > 0$ if $k$ is odd, otherwise $f_k(n-1) < 0$. Moreover, $f_{n+1}(n-1) = (-1)^n$ and so, if $n$ is even we have that $f_n(n-1) < 0$ and $f_{n+1}(n-1) > 0$ otherwise, $f_n(n-1) > 0$ and $f_{n+1}(n-1) < 0$. For $i=2$ we have that for $4 \leq k \leq n$, $f_k(n-2) = (-1)^k + (-1)^{k-2}\binom{k}{k-2}2 + (-1)^{k-1}\binom{k}{k-1}2$ . As $\binom{k}{k-2} > \binom{k}{k-1}$ then, $f_k(n-2) > 0$ if $k$ is even, otherwise $f_k(n-2) < 0$. Moreover, $f_{n+1}(n-2) = (-1)^{n-1}2n + (-1)^n2$. Therefore, if $n$ is even then $f_n(n-2) > 0$ and $f_{n+1}(n-2) < 0$, otherwise, $f_n(n-2) < 0$ and $f_{n+1}(n-2) > 0$. This proves the case $(a)$.\\
In order to prove case $(b)$, for fixed values of $k,i,p$, we denote by $s_p^kb_p^kc_p^k$ the term $(-1)^p\binom{k}{p}i!/(i+p-k)!$, where $s_p^k = (-1)^p$, $b_p^k = \binom{k}{p}$ and $c_p^k= i!/(i+p-k)!$. For $i=3$ and a fixed $k$ with $6\leq k \leq n$, we have that $b_{k-3}^k > b_{k-2}^k$ and $c_{k-3}^k \geq c_{k-2}^k$. As $s_{k-3}^k = s_{k-1}^k$ then, $f_k(n-3) > 0$ if $k$ is odd otherwise $f_k(n-3) < 0$. Moreover, $b_{n-2}^n > b_{n-1}^n$ and $s_{n-2}^n = s_{n}^n$ and thus, if $n$ is even then $f_{n+1}(n-3) > 0$ and $f_n(n-3) <0$ otherwise, $f_{n+1}(n-3) < 0$ and $f_n(n-3) >0$. For $i=4$ and $7 \leq k \leq n$ we have that $b_{k-4}^k > b_{k-3}^k$, $c_{k-4}^k \geq c_{k-3}^k$, $b_{k-2}^k > b_{k-1}^k$, and $c_{k-2}^k \geq c_{k-1}^k$. Thus, if $k$ is even then $f_k(n-4) > 0$ otherwise $f_k(n-4) < 0$. Moreover, $b_{n-3}^n > b_{n-2}^n$, $c_{n-3}^n \geq c_{n-2}^n$, $b_{n-1}^n > b_{n}^n$, and $c_{n-1}^n \geq c_{n}^n$. Therefore, if $n$ is even then $f_{n+1}(n-4) < 0$ and $f_n(n-4) > 0$ otherwise, $f_{n+1}(n-4) > 0$ and $f_n(n-4) < 0$. Finally, for $i=5$ and $8 \leq k \leq n$ we have that $b_{k-5}^k > b_{k-4}^k$, $c_{k-5}^k \geq c_{k-4}^k$, $b_{k-3}^k > b_{k-2}^k$, $c_{k-3}^k \geq c_{k-2}^k$, and $s_{k-5}^k = s_{k-3}^k = s_{k-1}^k$. Thus, if $k$ is odd $f_k(n-5) > 0$ otherwise, $f_k(n-5) < 0$. Now, $b_{n-4}^n > b_{n-3}^n$, $c_{n-4}^n \geq c_{n-3}^n$, $b_{n-2}^n > b_{n-1}^n$, $c_{n-2}^n \geq c_{n-1}^n$, and $s_{n-4}^n = s_{n-2}^n = s_{n}^n$. Therefore, if $n$ is even we have that $f_{n+1}(n-5) > 0$ and $f_n(n-5) < 0$ otherwise, $f_{n+1}(n-5) < 0$ and $f_n(n-5) > 0$ which ends the proof of the claim.
\end{proof}
By definition, $f_0(\_) = 1$. Now, when $i=0$, we have $f_k(n) = (-1)^k$ for $1 \leq k \leq n$, and $f_{n+1}(n) = 0$. Therefore, the sign sequence $\{\text{sign}(f_0(n)),\ldots,\text{sign}(f_{n+1}(n))\}$ is totally alternating. Thus, $s(n) = 0$. When $i=1,2$, by Claim \ref{cl1}, we need to compute $f_k(n-i)$ for $1 \leq k \leq i+1$. Clearly, $f_1(n-1) = 0$ and so $\text{sign}(f_1(n-1))$ is '-'. Moreover, $f_2(n-1) < 0$ and by Claim \ref{cl1}, $f_3(n-1) > 0$ which implies that $s(n-1) = 1$. For $i=2$, we can compute the values of $f_k(n-2)$ for $1 \leq k \leq 3$ which are the following: $f_1(n-2) > 0$, $f_2(n-2), f_3(n-2) < 0$ and, by Claim \ref{cl1}, $f_4(n-2) > 0$ which implies that $s(n-2) = 2$. For $i \in \{3,4,5\}$, by Claim \ref{cl1}, we need compute the values of $f_k(n-i)$ for $1 \leq k \leq i+2$. So, for $i=3$ we have that $f_1(n-3), f_2(n-3), f_4(n-3), f_5(n-3) > 0$ and $f_3(n-3) < 0$. By Claim \ref{cl1}, $f_6(n-3) < 0$ implies that $s(n-3) = 3$. For $i=4$ we have that $f_1(n-4),f_2(n-4),f_5(n-4),f_6(n-4) > 0$ and $f_3(n-4), f_4(n-4) < 0$. By Claim \ref{cl1}, $f_7(n-4) < 0$ implies that $s(n-4) = 4$. Finally, for $i=5$ we have that $f_1(n-5),f_2(n-5),f_3(n-5),f_6(n-5),f_7(n-5) > 0$ and $f_4(n-5),f_5(n-5) < 0$ and by Claim \ref{cl1}, $f_8(n-5) < 0$ which implies that $s(n-5) = 5$ and so, the lemma holds.
\end{proof}

As a direct consequence of Lemma \ref{lemaroot1} and Theorem~\ref{lemaroot2} we have the following result. 

\begin{corollary}
Let $n>1$. Then there is an eigenvalue $\lambda$ of the adjacency matrix of $\mathcal{S}_n$ such that $\lambda \in (n-2,n-1]$.
\end{corollary}

\begin{remark}
Let $n>1$ and let $\lambda_2$ the second largest eigenvalue of $\mathcal{S}_n$. We proved in Corollary \ref{coro:stellos} that $\lambda_2>n-\frac{4}{n-1}$. If $n\geq 5$, $n-\frac{4}{n-1}\geq n-1$, and therefore the eigenvalue $\lambda$ given by the previous corollary is not $\lambda_2$.
\end{remark}

\subsection{A high multiplicity eigenvalue}\label{sec:eigenstello}
Let $n\in \mathbb{N}$, $n\geq 2$. Label the leaves of $K_{1,n}$ with $\{1,\ldots,n\}$. For $i,j\in\{1,\ldots,n\}$, let $X{\substack{i\\j}}$ denote the set of brooms in $\mathcal{S}_n$ where $j$ is a descendant of $i$. Define the vector $\vec{z}\in\mathbb{R}^{V(\mc S_n)}$ by
\[
\vec{z}=\begin{cases}
1, &\text{if $v\in X{\substack{1\\2}}$,}\\
-1, &\text{if $v\in X{\substack{2\\1}}$,}\\
0, &\text{otherwise.}
\end{cases}
\]
Notice, in particular, that if $L$ is the set of brooms where both $1$ and $2$ are leaves, then $X{\substack{1\\2}},X{\substack{2\\1}},$ and $L$ partition the set of vertices in $\mathcal{S}_n$. Further, notice that $|L|=|V(\mathcal{S}_{n-2})|$ as deleting $1$ and $2$ from the broom gives a clear bijection between the two sets.

Recall that for every $k\in\mathbb{N}$, the entry $(A^k)_{ij}$ equals the number of walks of length $k$ from vertex $x_i$ to vertex $x_j$ in the graph. Therefore, by definition of $\vec{z}$, 
for every $T\in V(\mathcal{S}_n)$, $[A^k\vec{z}]_T$ equals the number of walks of length $k$ from $T$ to $X{\substack{1\\2}}$ minus the number of walks of length $k$ from $T$ to $X{\substack{2\\1}}$.

The map $\rho_{1,2}\colon K_{1,n}\to K_{1,n}$ defined by $\rho_{1,2}(1)=2$, $\rho_{1,2}(2)=1$, and $\rho_{1,2}(v)=v$ for $v\neq 1,2$ is a graph automorphism of $K_{1,n}$. By Remark~2.3 of~\cite{gargantini2026effect}, it induces a graph automorphism $\rho_{1,2}^*\colon\mathcal{S}_n\to\mathcal{S}_n$ whose effect on each $T\in V(\mathcal{S}_n)$ is to swap the labels $1$ and $2$. In particular, $\rho_{1,2}^*$ restricts to a bijection $X_{\substack{1\\2}}\leftrightarrow X_{\substack{2\\1}}$.

Two consequences follow. First, if $T_1$ and $T_2$ are the unique vertices in $X_{\substack{1\\2}}\cap N_1$ and $X_{\substack{2\\1}}\cap N_1$, respectively, then $[A^k\vec{z}]_{T_1}=-[A^k\vec{z}]_{T_2}$ for all $k$. Second, $[A^k\vec{z}]_T=0$ for every $T\in N_1\setminus(X_{\substack{1\\2}}\cup X_{\substack{2\\1}})$ and every $k\geq 1$, since the correspondence given by $\rho_{1,2}^*$ pairs walks from $T$ to $X_{\substack{1\\2}}$ bijectively with walks from $T$ to $X_{\substack{2\\1}}$.

Since $\vec{1}$ is an eigenvector for the largest eigenvalue of $A$, extend $\{\vec{1}\}$ to a basis of eigenvalues $\{\vec{1},\vec{v}_2,\ldots,\vec{v}_{|\mathcal{S}_n|}\}$. Notice that $\vec{z}\perp\vec{1}$. Let $\lambda_l$ be the eigenvalue of largest absolute value among those whose eigenspace is not perpendicular to $\vec{z}$.

By the power method, $A^k\vec{z}$ converges (up to rescaling by $\lambda_l^k$) to a linear combination of the eigenvectors associated to the eigenvalue $\lambda_l$ when $k\rightarrow\infty$. Since $[A^k\vec{z}]_T=0$ for all $k\geq 1$ and all $T\in N_1\setminus(X{\substack{1\\2}}\cup X{\substack{2\\1}})$, every eigenvector in this linear combination must also vanish on $N_1\setminus(X{\substack{1\\2}}\cup X{\substack{2\\1}})$.

Now repeat the construction for each $r\in\{2,\ldots,n\}$, defining $\vec{z}^{(r)}$ by placing $+1$ on $X{\substack{1\\r}}$, $-1$ on $X{\substack{r\\1}}$, and $0$ elsewhere. Each $\vec{z}^{(r)}$ is perpendicular to $\vec{1}$ and produces, via the same argument, a nonzero vector in the eigenspace of $\lambda_l$ or of $-\lambda_l$. The $n-1$ vectors $\vec{z}=\vec{z}^{(2)},\ldots,\vec{z}^{(n)}$ are linearly independent, since among these, $\vec{z}^{(r)}$ is the only one supported on $N_1\cap X_r$. Because each belongs to the eigenspace associated to either $\lambda_l$ or $-\lambda_l$, the sum of the dimensions of these two eigenspaces is at least $n-1$.

To give a lower bound for $|\lambda_l|$, notice that the Rayleigh quotient implies
\begin{align*}
    \frac{\vec{z}^TA\vec{z}}{\vec{z}^T\vec{z}}\leq |\lambda_l|.
\end{align*}
We estimate this quotient in several steps.
First,
\[
\vec{z}^T\vec{z}=|X{\substack{1\\2}}|+|X{\substack{2\\1}}|=|V(\mathcal{S}_n)|-|L|=|V(\mathcal{S}_n)|-|V(\mathcal{S}_{n-2})|.
\]
In order to estimate
$\vec{z}^TA\vec{z}$, notice that
\begin{align*}
    \vec{z}^TA\vec{z}=\sum_{v\in X{\substack{1\\2}}}\left(
|N(v)\cap X{\substack{1\\2}}|-|
N(v)\cap X{\substack{2\\1}}|\right)+\sum_{v\in X{\substack{2\\1}}}\left(
|N(v)\cap X{\substack{2\\1}}|-|
N(v)\cap X{\substack{1\\2}}|\right).
\end{align*}
Instead of counting the sizes of these sets, we relate them to the number of edges between  sets of vertices.
Thus, if $E(A,B)$ is the set of edges with one vertex in $A$ and the other vertex in $B$, then edges in $E(X{\substack{1\\2}},X{\substack{1\\2}})$ and in $E(X{\substack{2\\1}},X{\substack{2\\1}})$ are added twice, whereas edges in $E(X{\substack{1\\2}},X{\substack{2\\1}})$ are subtracted twice (once for each vertex in the edge). Hence, we get
\begin{align*}
\vec{z}^TA\vec{z}&=2\big(|E(X{\substack{1\\2}},X{\substack{1\\2}})|+|E(X{\substack{2\\1}},X{\substack{2\\1}})|-|E(X{\substack{1\\2}},X{\substack{2\\1}})\big)|.\end{align*}
Notice that $E(\mathcal{S}_n)$ is partitioned by the sets $E(X{\substack{1\\2}},X{\substack{1\\2}}), E(X{\substack{2\\1}},X{\substack{1\\2}}),E(X{\substack{2\\1}},X{\substack{2\\1}}),$ and $E(L,\mathcal{S}_n)$. Thus, 
\begin{align*}
\vec{z}^TA\vec{z}&=2|E(\mathcal{S}_n)|-2|E(L,\mathcal{S}_n)|-4|E(X{\substack{1\\2}},X{\substack{2\\1}})|.
\end{align*}
Notice that $2|E(L,\mathcal{S}_n)|=2|E(L,\mathcal{S}_n\setminus L)|+2|E(L,L)|$, and that  $|E(L,\mathcal{S}_n\setminus L)|+2|E(L,L)|$  equals the sum of the degrees of vertices in $L$, i.e., \[
|E(L,\mathcal{S}_n\setminus L)|+2|E(L,L)|=n|L|=n|\mathcal{S}_{n-2}|.\] On the other hand, the bijection between $\mathcal{S}_{n-2}$ and $L$ implies that $2E(L,L)$ equals the sum of the degrees of vertices in $\mathcal{S}_{n-2}$, i.e., $2E(L,L)=(n-2)|\mathcal{S}_{n-2}|$.
Thus,
\begin{align*}
2|E(L,\mathcal{S}_n)|&=2|E(L,\mathcal{S}_n\setminus L)|+2|E(L,L)|\\
&=2|E(L,\mathcal{S}_n\setminus L)|+4|E(L,L)|-2|E(L,L)|\\
&=2n|\mathcal{S}_{n-2}|-(n-2)|\mathcal{S}_{n-2}|\\
&=(n+2)|\mathcal{S}_{n-2}|.
\end{align*}
To estimate $|E(X{\substack{1\\2}},X{\substack{2\\1}})|$, notice that a vertex in $X{\substack{1\\2}}$ is adjacent to at most a vertex in $X{\substack{2\\1}}$, and this only happens if $2$ is the child of $1$. But the number of search trees where $2$ is a child of $1$ is $|V(\mathcal{S}_{n-1})|-|V(\mathcal{S}_{n-2})|$, as identifying $1$ and $2$ gives a clear injection to $V(\mathcal{S}_{n-1})$, further, the search trees in $V(\mathcal{S}_{n-1})$ where the vertex resulting of the identification is a leave can be removed. Therefore,
\begin{align*}
\vec{z}^TA\vec{z}&=2|E(\mathcal{S}_n)|-2|E(L,\mathcal{S}_n)|-4|E(X{\substack{1\\2}},X{\substack{2\\1}})|\\
    &\geq n|V(\mathcal{S}_n)|-(n+2)|V(\mathcal{S}_{n-2})|-2(|V(\mathcal{S}_{n-1})|-|V(\mathcal{S}_{n-2})|)\\
    &=n|V(\mathcal{S}_n)|-n|V(\mathcal{S}_{n-2})|-2|V(\mathcal{S}_{n-1})|.
\end{align*}
Using that 
\begin{align*}
&|V(\mathcal{S}_{n-1})|=1+(n-1)|V(\mathcal{S}_{n-2})|\\
&|V(\mathcal{S}_n)|=1+n|V(\mathcal{S}_{n-1})|=
1+n+n(n-1)|V(\mathcal{S}_{n-2})|,\\
\end{align*}
we get
\begin{align*}
    \vec{z}^TA\vec{z}&\geq n|V(\mathcal{S}_n)|-n|V(\mathcal{S}_{n-2})|-2|V(\mathcal{S}_{n-1})|\\
    &=n+n^2+n^2(n-1)|V(\mathcal{S}_{n-2})|-n|V(\mathcal{S}_{n-2})|-2-2(n-1)|V(\mathcal{S}_{n-2})|\\
    &=(n^3-n^2-3n+2)|V(\mathcal{S}_{n-2})|+n^2+n-2\\
    &\geq(n^3-n^2-3n+2)|V(\mathcal{S}_{n-2})|.
\end{align*}
Thus, using that $|V(\mathcal{S}_{n-2})|\geq n+1$ if $n\geq 4$,
\begin{align*}
    \lambda_l\geq \frac{\vec{z}^TA\vec{z}}{\vec{z}^T\vec{z}}&\geq
    \frac{(n^3-n^2-3n+2)|V(\mathcal{S}_{n-2})|}{1+n+(n^2-n-1)|V(\mathcal{S}_{n-2})|}\\
    &\geq \frac{(n^3-n^2-3n+2)|V(\mathcal{S}_{n-2})|}{(n^2-n)|V(\mathcal{S}_{n-2})|}\\
    &=n+\frac{-3n+2}{n^2-n}\\
    &=n+\frac{-3(n-1)-3+2}{n^2-n}\\
    &=n-\frac{3}{n}-\frac{1}{n^2-n}\\
    &>n-\frac{3}{n}.
\end{align*}
Therefore, we have proven the following.
\begin{theorem}
    $\mathcal{S}_n$ has an eigenvalue $\mu$, with $\lambda_1=n>|\mu|>  n-\frac{3}{n}-\frac{1}{n^2-n}$ and multiplicity at least $n-1$.
\end{theorem}

\subsection{Small eigenvalues of $\mathcal{S}_n$}
\label{sec:small_eigenvalue}

Notice that the vertices of $N_n$ can be seen as permutations on $\{1,\ldots,n\}$. In addition, observe that $N_n$ induces a bipartite subgraph of $\mc S_n$, with vertices represented with even permutations in one of the partite sets, which we denote by $X_n$, and vertices coming from odd permutations in the other, which we denote by $Y_n$. Finally, for each $v\in N_n$ let $p_v\in N_{n-1}$ be the only neighbor of $v$ in $N_{n-1}$ and $g_v$ the only neighbor of $p_v$ in $N_{n-2}$. Notice that $v,w\in N_{n}$  are adjacent if and only if $p_v,p_w$ are adjacent or  $g_v=g_w$. Thus, we denote by $X_{n-1}$ vertices in $N_{n-1}$ adjacent to vertices in $X_n$, and by $Y_{n-1}$ vertices in $N_{n-1}$ adjacent to vertices in $Y_n$.
Let $\epsilon\in \{\frac{1-\sqrt{5}}{2},\frac{1+\sqrt{5}}{2}\}$ and notice that $-\epsilon\cdot \epsilon=-1-\epsilon$. Consider the vector $\vec{x}$ defined as
\[
\vec{x}_v=\begin{cases}
\epsilon, &\text{if $v\in X_n$}\\
-\epsilon, &\text{if $v\in Y_n$}\\
-1, &\text{if $v\in X_{n-1}$}\\
1, &\text{if $v\in Y_{n-1}$}\\
0, &\text{otherwise.}\\
    \end{cases}
\]
Notice that each vertex in $N_{n-2}$ is adjacent to one vertex in $X_{n-1}$ and one vertex in $Y_{n-1}$. Note also that each vertex in $N_k$, with $k\leq n-3$, is adjacent to vertices $v$ such that $\vec{x}_v=0$. Thus, 
\begin{align*}
\left[A\vec{x}\right]_v=&\begin{cases}
-1-(n-1)\epsilon, &\text{if $v\in X_n$}\\
1+(n-1)\epsilon, &\text{if $v\in Y_n$}\\
(n-2)+\epsilon, &\text{if $v\in X_{n-1}$}\\
-(n-2)-\epsilon, &\text{if $v\in Y_{n-1}$}\\
0, &\text{otherwise}
\end{cases}
=\begin{cases}
(-1-\epsilon)-(n-2), &\text{if $v\in X_n$}\\
(1+\epsilon)+(n-2)\epsilon, &\text{if $v\in Y_n$}\\
(n-2)+\epsilon, &\text{if $v\in X_{n-1}$}\\
-(n-2)-\epsilon, &\text{if $v\in Y_{n-1}$}\\
0, &\text{otherwise}
\end{cases}\\
=&(-(n-2)-\epsilon)\begin{cases}
\epsilon, &\text{if $v\in X_n$}\\
-\epsilon, &\text{if $v\in Y_n$}\\
-1, &\text{if $v\in X_{n-1}$}\\
1, &\text{if $v\in Y_{n-1}$}\\
0, &\text{otherwise}\\
    \end{cases}\\
    =&(-(n-2)-\epsilon)
\vec{x}_v.
\end{align*}


\section{Conclusion}
\label{section_conclusions}

We have advanced the spectral theory of graph associahedra by establishing a lower bound on the second largest eigenvalue of $\A(G)$ and analyzing the spectrum of the stellohedron via equitable partitions—a novel technique in this context. In particular, for the stellohedron $\mc S_n$, we proved the existence of an eigenvalue in each interval $(n-i, n-i+1]$ for $1 \leq i \leq 5$, established an eigenvalue with high multiplicity in $(n - \frac{3}{n} + \frac{2}{n^2-n}, n)$, and identified two additional small eigenvalues, thereby offering new insights into the spectral structure of these polytopal graphs. Then, for the first two eigenvalues of $\mc S_n$, we have $\lambda_1 = n$ and $\lambda_2 > n - \frac{3}{n} + \frac{2}{n^2-n}$. 
In addition, our computational results suggest that the smallest eigenvalue of $\mc S_n$ is indeed $-(n-2)-\frac{1+\sqrt{5}}{2}$. It will be interesting to determine the exact value of the smallest eigenvalue of $\mc S_n$.

Another natural direction for further research is to explore the connection between the second largest eigenvalue of graph associahedra and its edge expansion, and consequently its mixing time. Previous work on expansion of $0/1$ polytopes \cite{FM-1991, Kai-2004, AGV-2018,ALGV-2024b}, has been extended to certain families of graph associahedra. Recently, Eppstein and Frishberg \cite{EF-2023} gave an upper and a lower bound for the edge expansion of the associahedron of, respectively, $\Omega(\frac{1}{\sqrt{n}\ \log n})$ and $O(\frac{1}{\sqrt{n}})$. Chang, Defant and Frishberg \cite{CDF-2025} established analogous bounds for the edge expansion of cyclohedra. Our analysis  shows that the edge expansion of stellohedron $\mc S_n$ is bounded above by $\frac{2}{n-1}$. This leaves the question of whether this asymptotic behavior holds for all graph associahedra.

Regarding the results provided in Subsection~\ref{LocEigen}, using a computer program, we have calculated the value of the functions $f_k(n-i)$ for $0 \leq k \leq n+1$ when $n=20$ obtaining $s(n-i) = i$ for $1 \leq i \leq 15$. The Table~\ref{tablaMario} shows the minimum length that a sign sequence must have in order to obtain $s(n-i) = i$.
\begin{table}[htbp]
\centering
\begin{tabular}{|c|c|c|c|c|c|c|c|c|c|c|c|c|c|c|c|}
\hline
$i$ & 1 & 2 & 3 & 4 & 5 & 6 & 7 & 8 & 9 & 10 & 11 & 12 & 13 & 14 & 15\\
\hline
$n \geq$ & 2 & 3 & 5 & 6 & 7 & 9 & 10 & 11 & 13 & 14 & 15 & 16 & 18 & 19 & 20\\
\hline
\end{tabular}
\caption{Values of $g(i)$.}
\label{tablaMario}
\vspace{-0.5 cm}
\end{table}

Notice that the length of the interval to obtain exactly $i$ signing agreements of consecutive members in a sequence of size $n+1$ is a function $g(i)$ greater than $i$ and such that $n \geq g(i)$.

In this regard, we get the following.

{\bf Open Problem:} For any $i \geq 1$ determine the function $g(i)$ such that in the interval $[0..g(i)]$ there are exactly $i$ sign agreements of consecutive members of the functions $f_0(n-i),\ldots,f_{g(i)}(n-i)$, where $\text{sign}(f_{g(i)}(n-i)) \neq \text{sign}(f_{g(i)+1}(n-i))$ and where the sign of the functions $f_k(n-i)$ with $g(i) < k \leq n+1$ is full alternating. 





\bibliographystyle{elsarticle-num}
\bibliography{bibliography-1}

\end{document}